\documentclass[12pt,twoside,reqno]{amsart}

\usepackage{hyperref}
\usepackage{amsmath}
\usepackage{amsthm}
\usepackage{amsfonts}
\usepackage{amssymb}

\def\R{{\mathbb R}}
\def\N{{\mathbb N}}

\def\Z{{\mathbb Z}}
\def\KK{{\mathcal{K}}}

\def\d{\delta}

\def\e{\varepsilon}

\def\a{\alpha}

\def\l{\lambda}
\def\G{\Gamma}
\def\L{\Lambda}

\newtheorem{theorem}{Theorem}[section]
\newtheorem*{theorem*}{Theorem}
\newtheorem{lemma}[theorem]{Lemma}

\newtheorem*{proposition*}{Proposition}

\numberwithin{equation}{section}

\begin{document}

\title {A ``rare'' plane set with Hausdorff dimension 2}

\author{Vladimir Eiderman and Michael Larsen}
\address{Vladimir Eiderman, Department of  Mathematics, Indiana University, Bloomington, IN}
\email{veiderma@indiana.edu}
\address{Michael Larsen,  Department of  Mathematics, Indiana University, Bloomington, IN}
 \email{mjlarsen@indiana.edu}
\thanks{ML was partially supported by NSF grant DMS-1702152.}

\begin{abstract}
We prove that for every at most countable family $\{f_k(x)\}$ of real functions on $[0,1)$ there is a single-valued real function $F(x)$, $x\in[0,1)$, such that the Hausdorff dimension of the graph $\G$ of $F(x)$ equals 2, and for every $C\in\R$ and every $k$, the intersection of $\G$ with the graph of the function $f_k(x)+C$ consists of at most one point. We also construct a family of functions of cardinality continuum and a function $F$ with  similar properties.
\end{abstract}
\maketitle

\section{Introduction}\label{introd}

The motivation of this note comes from the following question by Sergei Treil (August 2018, private communication). {\it Let $E$ be a set in $\R^n$, and let $K$ be an $n$-dimensional cone in $\R^n$. Suppose that for every line $l$ in $K$ and for every vector $\mathbf{b}$, the intersection $E\cap(l+\mathbf{b})$ is at most countable. Does it follow that the Hausdorff dimension of $E$ is less than $n$?} For $n=2$, an affirmative answer was given by Marstrand \cite{Mar} under the additional assumptions that $E$ is measurable with respect to $s$-dimensional Hausdorff measure $H^s$, and $0<H^s(E)<\infty$. Marstrand proved that if $0<s\le2$, then at $H^s$-almost all points $x\in E$ the following is true: for almost all straight lines $l$ passing through $x$, $H^{s-1}(E\cap l)<\infty$ and the Hausdorff dimension of $E\cap l$ is equal to $s-1$. See \cite{Ma1}, \cite{Ma2} and references therein for generalizations and related results.

We consider the case $n=2$ and try to approach the question from the other end: for which sets $K$ of directions (not necessarily $n$-dimensional) is the answer negative? Marstrand's theorem quoted above implies only that under additional assumptions, $K$ has zero measure. T.~Orponen obtained the estimate of the set of exceptional directions in terms of Hausdorff dimension in a much more general setting -- see \cite[Theorem~5.2]{Ma2}. Our goal is to construct examples which provide us with more detailed information about $K$ and intersections $E\cap(l+\mathbf{b})$. The case when $K$ consists of only one direction is known---see for example \cite{1}. Namely, there exists a function $F(x)$ (which can even be continuous!) whose graph has Hausdorff dimension 2. So, the intersection of the graph of $F(x)$ with every vertical line consists of at most one point.

We show that the answer to  Treil's question is negative in the case of any countable set of directions. In fact we prove a much more general assertion.

\begin{theorem}\label{th11} For every at most countable family $\mathcal F$ of real functions on $[0,1)$ there is a (single-valued) function $F(x)$, $x\in[0,1)$, such that

(i) the Hausdorff dimension of the graph $\G$ of $F(x)$ equals 2;

(ii) the intersection of $\G$ with the graph of any function $f_k(x)+C$, where $f_k(x)\in\mathcal F$, $C\in\R$, consists of at most one point.
\end{theorem}

Moreover, there are uncountable families of directions and functions $F$ with similar properties -- see Section~3 for a more general result.

Recall that the Hausdorff measure $H^s(E)$, $s\ge0$, and the Hausdorff dimension dim${}_H(E)$ of a set $E$ are defined by the equalities
\begin{align*}
H^s(E)& =\lim_{\d\to0}\inf_{r_i<\d}\sum r_i^s.\\
\text{dim}_H(E)& =\sup\{s:H^s(E)=\infty\}=\inf\{s:H^s(E)=0\},
\end{align*}
where the $\inf_{r_i<\d}$ is taken over all at most countable covers of $E$ by disks with radii $r_i<\d$.

T.~Keleti \cite{2} constructed a compact subset of $\R$ with Hausdorff dimension 1
that intersects each of its non-identical translates in at most one point. This result and our  theorem have a similar flavor, but their proofs are completely different.

The authors are grateful to Professor Sergei Treil for very useful discussion.

\section{Proof of Theorem \ref{th11}}\label{sect2}

Every $x\in\R$ can be written in the form
$$
x=\lfloor x\rfloor+\{x\}=\lfloor x\rfloor+\sum_{i=1}^\infty x_i2^{-i},
$$
where $\lfloor x\rfloor,\ \{x\}$ are the integer and the fractional parts of $x$ correspondingly, and each $x_i$ is either 0 or 1. We write $(0100\cdots0\cdots)$ instead of $(0011\cdots1\cdots)$. In other words, the binary expansion of every number $\{x\}$ in $[0,1)$ contains infinitely many zeros. Such a representation is unique.

We choose a set $T$ of positive integers and a collection of 3-element sets $S_{ij}$ of positive integers indexed by ordered pairs $(i,j)$ of positive integers in such a way that the following statements hold:

\begin{enumerate}
\item The density of $T$ in the positive integers is $1$.
\item Each $S_{ij}$ is of the form  $\{s_{ij},s_{ij}+1,s_{ij}+2\}$ for some positive integer $s_{ij}$.
\item All sets $S_{ij}$ and $T$ are mutually disjoint.
\end{enumerate}

If $x = \sum_i x_i 2^{-i} \in [0,1)$ and $s$ is a positive integer, we define
\begin{equation}
\label{g-def}
g_s(x) := x_s 2^{-s} + x_{s+1} 2^{-s-1}.
\end{equation}
We extend $g_s(x)$ to a function on $\R$ by imposing periodicity: $g_s(x+1)=g_s(x)$. In other words, we set $g(x):=g(\{x\})$.

\begin{lemma}
\label{reading}
Let $s$ be a positive integer,
$U$ a subset of the positive integers which is disjoint from $\{s,s+1,s+2\}$, and $a\in \R$.
Let
$$A := g_s(a)+\sum_{i\in U} 2^{-i} - a;\ \ B:= g_s(2^{-s}+a)+\sum_{i\in U} 2^{-i} - a.$$
Then
$$\{ 2^{s-1} A \} \in [0,1/8]\cup[3/4,1);\quad \{ 2^{s-1} B \} \in [1/4,5/8].$$
\end{lemma}

\begin{proof}
We partition $U$ into $U^+ := U\cap [1,s-1]$ and $U^- := U\cap [s+3,\infty)$.  Thus
$$\sum_{i\in U} 2^{-i} = \sum_{i\in U^+} 2^{-i} + \sum_{i\in U^-} 2^{-i} = \frac{m}{2^{s-1}} + \delta$$
for some $m\in\Z$, $\delta \in [0,2^{-s-2}]$.  As $\{ 2^{s-1} A \}$ and $\{ 2^{s-1} B \}$
only depend on $\{a\}$ and $U^-$, we may assume $a\in [0,1)$.
We can therefore write

$$
a - g_s(a) = \sum_{i\in A^+} 2^{-i} + \sum_{i\in A^-} 2^{-i},
$$
where $A^+\subset [1,s-1]$ and $A^-\subset [s+2,\infty)$ are sets of integers.
Thus,

$$a-g_s(a) = \frac n{2^{s-1}}+\e$$
for some $n\in\Z$, $\e \in [0,2^{-s-1}]$.    It follows that
$$2^{s-1}A = m - n + 2^{s-1}(\delta-\e)\in [m-n-1/4,m-n+1/8].$$
Likewise,
$$a - g_s(2^{-s} + a)  = (2^{-s}+a)  - g_s(2^{-s} + a) - 2^{-s}= \frac{n-1/2}{2^{s-1}}+\e$$
for some integer $n$ and $\e \in [0,2^{-s-1}]$, so
$$2^{s-1}B = m - n + 1/2 + 2^{s-1}(\delta-\e)\in [m-n+1/4,m-n+5/8].$$
Lemma \ref{reading} is proved.
\end{proof}

For positive integers $i$ and $j$, we define
$$h_{ij}(x) = g_{s_{ij}}(f_i(x) + x_j 2^{-s_{ij}}),\quad f_i\in\mathcal F.$$
Define $F(x)$, $x\in[0,1)$, by the equality
\begin{equation}\label{f1}
F(x)=F\bigg(\sum_{i=1}^\infty x_i2^{-i}\bigg)=\sum_{i\in T} x_{i^2}2^{-i}+\sum_{i=1}^\infty \sum_{j=1}^\infty h_{ij}(x).
\end{equation}

\begin{lemma}
\label{two}
The function $F(x) - f_i (x)$ is one-to-one on $[0,1)$ for every $f_i\in\mathcal F$.
\end{lemma}

\begin{proof}
Fix $i$, $j$, and $x\in [0,1)$, and observe that $h_{ij}(x)$ is a sum of $2^{-k}$ as $k$ ranges
over some subset of $\{s_{ij},s_{ij}+1\}$, while $\sum_{k\in T} x_{k^2} 2^{-k}$ and $\sum_{(k,l)\ne (i,j)} h_{kl}(x)$ are sums of $2^{-k}$ over some subsets of $T$ and of $\N\setminus T$ which are both disjoint from  $S_{ij}$. Thus,
$$
F(x)-f_i(x) = \begin{cases}g_{s_{ij}}(f_i(x)) + \sum_{i\in U} 2^{-i}-f_i(x), &x_j=0,\\
g_{s_{ij}}(f_i(x)+2^{-s_{ij}}) + \sum_{i\in U} 2^{-i}-f_i(x), & x_j=1,
\end{cases}$$
where $U$ is a set of positive integers which is disjoint from $S_{ij}$.

Choose $x\ne y$. There exists $j$ such that $x_j\ne y_j$. According to Lemma~\ref{reading},
$$\{2^{s_{ij}-1}(F(x)-f_i(x))\}\ne\{2^{s_{ij}-1}(F(y)-f_i(y))\}$$
for every positive integer $i$. Hence, $F(x)-f_i(x)\ne F(y)-f_i(y)$.
\end{proof}
Note that Lemma~\ref{two} implies (ii) of the main theorem.

The following lemma establishes the validity of (i).

\begin{lemma}\label{le1}
Let $T$ and $S_{ij}$ be defined as above, and for each pair of positive integers $(i,j)$, let $h_{ij}\colon \R\to \R$
denote any function whose range is contained in
$$\{k2^{-s_{ij}-1}: k=0,1,2,3\}.$$
Let $\G$ denote the graph of $F(x)$ defined as in (\ref{f1}).  Then
$$
\dim_H(\G)=2.
$$
\end{lemma}

The proof of Lemma \ref{le1} is based on the following assertion.

\begin{lemma}\label{le2} Let functions $h_{ij}(x)$ be chosen as above. For every $\a<2$ and $\e>0$ there exists $\d=\d(\a,\e)>0$ with the following property. For every disk $D(r)$ with radius $r<\d$, the length of the projection of $\G\cap D(r)$ onto the $x$-axis is less than $\e r^\a$.
\end{lemma}

Let us show that Lemma \ref{le2} implies Lemma \ref{le1}.

\begin{proof}[Proof of Lemma \ref{le1}] Suppose that $\text{dim}_H(\G)<2$. Choose $\beta$ so that $\text{dim}_H(\G)<\beta<2$. Let $\d:=\d(\beta,1)$ be the number in Lemma \ref{le2}. There is an at most countable family of disks $D_i(r_i)$ such that $r_i<\d$, $\G\subset\bigcup_iD_i(r_i)$, and $\sum_ir_i^\beta<1$. We have
$$
|\text{Pr}(\G)|\le\sum_i|\text{Pr}(\G\cap D_i(r_i))|\le\sum_ir_i^\beta<1,
$$
where $|\text{Pr}(A)|$ denotes the length of the projection of a set $A$ onto the $x$-axis. Since $\G$ projects onto the whole of $[0,1)$, this contradiction proves that $\text{dim}_H(\G)=2$.
\end{proof}

\begin{proof}[Proof of Lemma \ref{le2}]
By (\ref{f1}), for every $x\in[0,1)$, the value $y = F(x)$ can be written as $\sum_{i\in U} 2^{-i}$ for some set $U$ of positive integers which does not contain any integer of the form $s_{ij}+2$ but does contain $i$ whenever $i\in T$ and $x_{i^2} = 1$.  Since there are infinitely many integers of the form $s_{ij}+2$, we have
$y_i = 1$ if and only if $i\in U$. Hence, $y_i = x_{i^2}$ for $i\in T$.

We may assume that $\d<1/2$. Let $N$ be such that $2^{-N-1}\le r<2^{-N}$. Then a disk $D(r)$ intersects at most nine dyadic squares with side length $2^{-N}$. Hence, it suffices to prove the existence of $N_0$ such that for every open dyadic square $Q$ with side length less than $2^{-N}$, where $N>N_0$,
$$
|\text{Pr}(Q\cap\G)|<\e2^{-N\a},\quad N>N_0.
$$
Fix $Q$. For all points $(x,y)\in Q$, the first $N$ digits $x_1,\dots,x_N$ in the binary representations of $x$ are determined by $Q$, and likewise for $y_1,\ldots,y_N$. Let $M=M(N)$ be the number of positive integers in the set $[1,N]\cap T$.
By (\ref{f1}), for all $(x,y)\in \Gamma$ and $n\in T$, we have $x_{n^2} = y_n$.  Therefore, for all $(x,y)\in \Gamma\cap Q$, $x_m$ is constant for $m\in [1,N]$ and also for $m\in \{n^2: n\in T\cap [1,N]\}$.
The union of these two sets has at least $M+N-\sqrt N$ elements.
Since $\lim_{N\to\infty}M/N=1$ and $\a < 2$, we have
$$|\text{Pr}(Q\cap\G)|\le2^{-(N+M-\sqrt N)}=2^{-N(1+M/N-\a-\sqrt N/N)}2^{-N\a}<\e2^{-N\a},$$
if $N$ is sufficiently large. Lemma \ref{le2} is proved
\end{proof}

\section{Families $\mathcal F$ of cardinality continuum}\label{sect3}

\begin{theorem}\label{th31} For every real function $f(x)$ on [0,1) there exist a set $\mathcal K$ of real numbers and a single-valued function $F:[0,1)\to[0,1)$ such that

(i) $\KK$ has cardinality continuum;

(ii) the Hausdorff dimension of the graph $\G$ of $F(x)$ equals 2;

(iii) the intersection of $\G$ with the graph of any function $kf(x)+C$, where $k\in\mathcal K$, $C\in\R$, consists of at most one point.
\end{theorem}

\begin{proof} We may assume that $0\not\in\KK$. Let
$$
\L:=\{\l:\l=1/k,\ k\in\KK\}.
$$
Then (iii) in Theorem \ref{th31} is equivalent to the following statement: the function $\l F(x)-f(x)$, $x\in[0,1)$, is one-to-one for every $\l\in\L$.

Let
$$
s_0=5,\quad s_j=(j+1)s_{j-1},\ j\ge1,\quad V:=[0,5]\cup\bigg(\bigcup_{j=1}^\infty[s_j-s_{j-1},s_j+s_{j-1}]\bigg),\quad T:=\N\setminus V.
$$
Note that the density of $T$ in the positive integers is 1, and $i\ge16$ for $i\in T$.

Define $F(x),\ x\in[0,1)$, by the equality
\begin{equation}
\label{f2}
F(x)=F\bigg(\sum_{i=1}^\infty x_i2^{-i}\bigg)=\sum_{i\in T} x_{i^2}2^{-i}+\sum_{j=1}^\infty h_{j}(x),
\end{equation}
where $h_{j}(x) = g_{s_{j}}(f(x) + x_j 2^{-s_{j}})$, and $g_s$ is defined by (\ref{g-def}).
Essentially the same arguments as in the proof of Lemmas \ref{le1} and \ref{le2} yield (ii). Let
$$
\L:=\{\l\in\R:\l=1+\sum_{i=1}^\infty\l_i2^{-s_{i}}\},
$$
where each $\l_i$ is either 0 or 1. Obviously, $\L$ (and hence $\KK$) has cardinality continuum.

To establish (iii), it suffices to prove the following claim: for every $x\in[0,1)$, $j\in\N$, and $\l\in\L$, there exists a set $U=U(x,j,\l)$ of nonnegative integers which is disjoint from $\{s_j,s_j+1,s_j+2\}$ and such that
\begin{equation}\label{lambda-F}
\l F(x)=h_{j}(x)+\sum_{i\in U}2^{-i}.
\end{equation}
Indeed, the claim implies that $\l F(x)-f(x)$ is one-to-one for every $\l\in\L$ exactly by the same arguments as in the proof of Lemma \ref{two}.

We fix $x\in[0,1)$, $j\in\N$, and $\l\in\L$, and split $T$ into $T^-=T\cap[0,s_j-s_{j-1}-1]$ and
$T^+=T\cap[s_j+s_{j-1}+1,\infty)$ (for $j=1$, the set $T^-$ is empty). Thus,
\begin{align*}
\l F(x)& =F(x)+\bigg(\sum_{k=1}^\infty\l_k2^{-s_{k}}\bigg)\cdot\bigg(\sum_{i\in T} x_{i^2}2^{-i}+\sum_{i=1}^\infty h_{i}(x)\bigg)\\
& =F(x)+\sum_{k=1}^{j-1}\sum_{i\in T^-}\l_k x_{i^2}2^{-(s_{k}+i)}+
\sum_{k=1}^{j-1}\sum_{i\in T^+}\l_k x_{i^2}2^{-(s_{k}+i)}\\
&+\sum_{k=j}^\infty\sum_{i\in T}\l_k x_{i^2}2^{-(s_{k}+i)}+\sum_{k=1}^\infty\sum_{i=1}^\infty\l_k 2^{-s_k}h_{i}(x),\\
\end{align*}
which we write $F(x)+\Sigma_1+\Sigma_2+\Sigma_3+\Sigma_4$.
Note that for $j=1$,  $\Sigma_1=\Sigma_2=0$.

By \eqref{f2}, we may write $F(x)$ in the form
$$
F(x)=\sum_{i\in I_F^-}2^{-i}+h_j(x)+\sum_{i\in I_F^+}2^{-i},
$$
where $I_F^-\subset[0,s_j-s_{j-1}-1],\ I_F^+\subset[s_j+s_{j-1}+1,\infty)$; for $j=1$, the set $I_F^-$ is empty.

For the sum $\Sigma_1$ we have $s_{k}+i\le s_{j-1}+s_j-s_{j-1}-1=s_{j}-1$. Hence,
$$
\Sigma_1=\sum_{i\in I_1}2^{-i},\quad I_1\subset[0,s_{j-1}].
$$
As $T^+\subset [s_j+s_{j-1}-1,\infty)$,
$$
\Sigma_2<\sum_{k=1}^{j-1}2^{-(s_{k}+s_j+s_{j-1})}<2^{-(s_j+s_{j-1})}.
$$
Since $i\ge16$ for $i\in T$,
$$
\Sigma_3<\sum_{k=j}^{\infty}2^{-(s_{k}+15)}<2^{-(s_j+14)}.
$$
Finally, $\Sigma_4$ may be decomposed into two sums. The first sum
$$
\Sigma_{4,1}=\sum_{k=1}^{j-1}\sum_{i=1}^{j-1}\l_k 2^{-s_k}h_{i}(x)=\sum_{i\in I_4^-}2^{-i},\quad I_4^-\subset[0,s_{j}),
$$
since $s_{k}+s_i+1<s_{j}$. In the second sum at least one index $k$ or $i$ is greater than or equal to $j$. If $i\ge j$, we have
\begin{align*}
\sum_{k=1}^\infty\sum_{i=j}^\infty\l_k 2^{-s_k}h_{i}(x)&\le\sum_{k=1}^\infty\sum_{i=j}^\infty2^{-s_k}(2^{-s_i}+2^{-s_i-1})\\
&<\sum_{k=1}^\infty2^{-s_k}\cdot2^{-s_j+2}\le2^{-s_j-7},
\end{align*}
since $s_k\le10$ as $k\ge1$. The case $k\ge j$ is analogous to the previous one.

Combining these estimates, we obtain \eqref{lambda-F}, which implies the theorem.
\end{proof}


\begin{thebibliography}{XX}

\bibitem{1} F.~Bayart and Y.~Heurteaux, {\it On the Hausdorff Dimension of Graphs
of Prevalent Continuous Functions on Compact Sets}, J. Barral and S. Seuret (eds.), Further Developments in Fractals and Related Fields,
Trends in Mathematics, DOI 10.1007/978-0-8176-8400-6-2, Springer Science+Business Media New York 2013, 25--34.

\bibitem{2} T.~Keleti, {\it A 1-dimensional subset of the reals that intersects each of its translates in at most a single point}, Real Anal. Exchange {\bf24} (1998/99), no. 2, 843--844.

\bibitem{Mar}   J.~M.~Marstrand, {\it Some  fundamental geometrical properties  of  plane  sets  of fractional dimensions}, Proc. London Math. Soc (3) {\bf4} (1954), 257--302.

\bibitem{Ma1} P.~Mattila, {\it  Hausdorff dimension, orthogonal projections and intersections with planes}, Ann. Acad. Sci. Fenn. A Math. {\bf 1} (1975), 227--244.

\bibitem{Ma2} P.~Mattila, {\it Hausdorff dimension, projections, intersections, and Besicovitch sets}, (2018) arXiv:1712.09199.

\end{thebibliography}
\end{document}